\documentclass[a4paper,12pt]{article}

\usepackage{amsmath,euscript}
\usepackage{amssymb}
\usepackage{amsthm}
\usepackage{enumerate}
\usepackage{amsfonts}
\usepackage{amssymb,amsmath,array}

\numberwithin{equation}{section}

\usepackage[top=2cm, bottom=2cm, left=2cm, right=2cm]{geometry}
\usepackage{graphicx, array, amsmath, amsthm, mathrsfs, framed, stmaryrd, marvosym, bbm, amssymb, mathtools, verbatim, tikz}
\usepackage{a4wide} 
\allowdisplaybreaks[1]

\usepackage[T1]{fontenc}

\input xy
\xyoption{matrix}\xyoption{arrow}
\def\edge{\ar@{-}}
\def\dedge{\ar@{.}}

\theoremstyle{plain}
\newtheorem{theorem}{Theorem}[section]
\newtheorem{proposition}[theorem]{Proposition}

\newtheorem{lemma}[theorem]{Lemma}
\newtheorem{corollary}[theorem]{Corollary}

\theoremstyle{definition}
\newtheorem{definition}[theorem]{Definition}

\newtheorem{example}[theorem]{Example}
\newtheorem{notation}[theorem]{Notation}

\def\widebar{\overline}

\def\ch{{\mathcal H}}
\def\spec{{\rm Spec}}
\def\hspec{\ch\text{-}\spec} 

\def\mn{{\mathbb N}}
\def\mr{{\mathbb R}}
\def\mz{{\mathbb Z}}

\DeclareMathOperator{\Tdeg}{{\rm Tdeg}}

\def\mz{{\mathbb Z}}

\newcommand{\gc}{ [ \hspace{-0.65mm} [}
\newcommand{\dc}{]  \hspace{-0.65mm} ]}

\newcommand \hp {\vartriangleleft'_\ch}

\def\ch{{\mathcal H}}
\def\co{{\mathcal O}}

\def\oqmm13{\co_q(M_{1,3})}
\def\oqm23{\co_q(M_{2,3})}

\newcommand{\la}{\lambda}

\newcommand{\id}{\text{id}}

\newcommand{\be}{\begin{equation}}
\newcommand{\ee}{\end{equation}}

\newcommand{\ol}{\overline}

\newcommand\mk{{\mathbb K}}

\newcommand\gk{{\rm GK}}
\newcommand\height{{\rm ht}}

\newcommand\Spec{{\rm Spec}}

\DeclareMathOperator\Fract{Fract}
\newcommand\Rbar{\overline{R}}
\newcommand\sighat{\widehat{\sigma}}
\newcommand\tauhat{\widehat{\tau}}
\newcommand\bfq{\mathbf{q}}
\newcommand\ce{\mathcal{E}}
\newcommand\cf{\mathcal{F}}
\newcommand\Obfq{{\mathcal O}_\bfq}

\DeclareMathOperator\Ld{Ld}

\newcommand\gfrak{{\frak g}}
\newcommand\mfrak{{\frak m}}
\DeclareMathOperator\chr{char}

\begin{document}

\title{Tauvel's height formula for 
quantum nilpotent algebras}
\author{K R Goodearl\thanks{\,The research of the first named author was supported
by US National Science Foundation grant DMS-1601184.},~~S Launois\thanks{\,The research of the second named author was supported by EPSRC grant EP/N034449/1.}~~and
T H Lenagan\thanks{\,The research of the third named author was partially supported by a Leverhulme Trust Emeritus Fellowship.}}
\date{}

\maketitle

\begin{abstract} Tauvel's height formula, which provides a link between the height of a prime ideal and the Gelfand-Kirillov dimension of the corresponding factor algebra, is 
verified for quantum nilpotent algebras. 
\end{abstract}


\section{Introduction}

Quantum nilpotent algebras, originally introduced under the name 
 Cauchon-Goodearl-Letzter extensions \cite{llr-ufd}, are iterated Ore extensions with special properties which cover a wide variety of algebras, including 
many of the algebras that appear as quantised coordinate rings. Examples include 
quantum Schubert cell algebras, quantum matrix algebras, generic quantised coordinate rings
of affine, symplectic and euclidean spaces, and generic
quantised Weyl algebras.
The precise definition of a quantum nilpotent algebra is recalled in Section~\ref{section-cgl}. It is designed to allow application of both Cauchon's deleting derivations 
algorithm and the Goodearl-Letzter stratification theory for prime ideals. 

In studying the prime spectrum of an algebra $R$, key invariants for a prime ideal 
$P$ are its height, $\height(P)$, and  the Gelfand-Kirillov dimension of the 
factor algebra, $\gk(R/P)$. Tauvel's height formula 
\[
\gk(R/P)+\height(P) = \gk(R)
\]
provides a useful connection between these two invariants and it is of interest 
to know when this formula holds. 
The purpose of this paper is to show that Tauvel's height formula does hold for 
all quantum nilpotent algebras. 
 
Several  verifications of the height formula for particular classes of algebras have 
 proceeded by first proving that the algebras are catenary, see, for example, 
 \cite{c2, gl-cat, h1, llr2, oh, y1, y2}. In these papers, catenarity is demonstrated by first establishing 
 certain homological conditions and showing that normal separation holds 
 for the prime spectrum. While the homological conditions can 
easily be established for quantum nilpotent algebras, normal separation remains elusive at the moment, although we do conjecture that this condition holds for all quantum nilpotent algebras. 
 
 The approach taken in the present paper is to exploit Cauchon's deleting derivations algorithm to establish Tauvel's height formula for torus-invariant prime ideals of a quantum nilpotent algebra, then extend this to primitive ideals by using the Goodearl-Letzter 
 stratification theory by virtue of the fact that the primitive ideals are identifiable 
 as the maximal members in individual strata. Finally, the formula is established for arbitrary prime ideals via the link between the prime spectrum of a given stratum and the prime spectrum of an associated commutative Laurent polynomial algebra, where the formula is well-known. As far as we are aware, the approach we use (from torus-invariant primes to primitive ideals, then arbitrary primes)
has not been used before; so this result advertises the approach.



\section{Quantum nilpotent algebras and stratification theory}\label{section-cgl} 

Fix a base field $\mk$ throughout the paper. Let $N$ denote a positive integer and
 $R$ an iterated Ore extension of the form
\begin{equation}
R\ = \ \mk[x_1][x_2;\sigma_2,\delta_2]\cdots[x_N;\sigma_N,\delta_N],
\end{equation}
 where $\sigma_j$ is an automorphism of the $\mk$-algebra $R_{j-1}:=\mk[x_1][x_2;\sigma_2,\delta_2]\dots[x_{j-1};\sigma_{j-1},\delta_{j-1}]$
 and  $\delta_j$ is a $\mk$-linear $\sigma_j$-derivation of
 $R_{j-1}$ for all $j\in \gc 2 ,N \dc$. (When needed, we denote $R_0 := \mk$ and set $\sigma_1 := \id_\mk$, $\delta_1 := 0$.) In other words, $R$ is a skew polynomial ring whose multiplication is determined by:
 $$x_j a = \sigma_j(a) x_j + \delta_j(a)$$
 for all $j\in \gc 2 ,N \dc$ and $a \in R_{j-1}$. Thus $R$ is a noetherian domain.  Henceforth, we
 assume that $R$ is a quantum nilpotent algebra, as in the following definition. 
 
 \begin{definition}  \label{CGLdef}
The iterated Ore extension $R$ is said to be a \emph{quantum nilpotent algebra} or a \emph{CGL extension}
\cite[Definition 3.1]{llr-ufd} if it is 
equipped with a rational action of a $\mk$-torus $\ch = (\mk^*)^d$ 
by $\mk$-algebra automorphisms satisfying the following conditions:
\begin{enumerate}
\item[(i)]  The elements $x_1, \ldots, x_N$ are $\ch$-eigenvectors.
\item[(ii)] For every $j \in \gc 2,N \dc$, $\delta_j$ is a locally nilpotent 
$\sigma_j$-derivation of $R_{j-1}$. 
\item[(iii)] For every $j \in \gc 1,N \dc$, there exists $h_j \in \ch$ such that $(h_j\cdot)|_{R_{j-1}} = \sigma_j$ and
$h_j \cdot x_j = q_j x_j$ for some $q_j \in \mk^*$ which is 
not a root of unity. 
\end{enumerate}
(We have omitted the condition $\sigma_j \circ \delta_j = q_j \delta_j \sigma_j$ from the original definition, as it follows from the other conditions; see, e.g., \cite[Eq. (3.1); comments, p.694]{gy0}.) From (i) and (iii), there exist scalars $\la_{j,i} \in \mk^*$ such that $\sigma_j(x_i) = \la_{j,i} x_i$ for all $i < j$ in $\gc1,N \dc$.
\end{definition}

A two-sided ideal $I$ of $R$ is said to be {\em $\ch$-invariant} if $h\cdot I=I$ for all
$h \in \ch$.  An {\em $\ch$-prime ideal} of $R$ is a proper $\ch$-invariant ideal
$J$ of $R$ such that if  $J$ contains the product of two
$\ch$-invariant ideals of $R$ then $J$ contains at least one of them. We denote
by $\ch$-$\spec(R)$ the set of all $\ch$-prime ideals of $R$. Observe
that if $P$ is a prime ideal of $R$ then
\begin{equation}
(P:\ch)\ := \ \bigcap_{h\in \ch} h\cdot P
\end{equation}
(namely, the largest $\ch$-invariant ideal contained in $P$) is an $\ch$-prime ideal of
$R$. For any $\ch$-prime ideal $J$ of $R$, we denote
by $\spec_J (R)$ the {\em $\ch$-stratum} associated  to $J$; that is,  
\begin{equation}
\spec_J (R) :=\{ P \in \spec(R) \mbox{ $\mid$ } (P:\ch)=J \}.
\end{equation}
Then the $\ch$-strata of $\spec(R)$ form a partition of $\spec(R)$; that
is,
\begin{equation}
\label{eq:Hstratification}
 \spec(R) := \bigsqcup_{J \in \ch\mbox{-}\spec(R)}\spec_J(R).
 \end{equation}
This partition is the so-called {\em $\ch$-stratification} of $\spec(R)$. 

It follows from work of Goodearl and Letzter \cite[Proposition 4.2]{gl2} that every $\ch$-prime ideal of $R$ is completely prime, so $\ch$-$\spec(R)$ coincides with the set of $\ch$-invariant completely prime ideals of $R$. Moreover there are at most $2^N$ $\ch$-prime ideals in $R$. As a consequence,  the prime spectrum
of $R$ is partitioned into a finite number of parts, the $\ch$-strata. In case $R$ is \emph{torsionfree}, meaning that  the subgroup $\langle \lambda_{j,i} \mid 1\le i<j \le N \rangle$ of $\mk^*$ is torsionfree, all prime ideals of $R$ are completely prime \cite[Theorem 2.3]{gl3}.

For each $\ch$-prime ideal $J$ of $R$, the space $\spec_J(R)$ (equipped with the relative Zariski topology inherited from $\spec(R)$) is homeomorphic to ${\rm Spec}(\mk[z_1^{\pm 1},\ldots ,z_d^{\pm 1}])$ for some $d$ which depends on $J$ \cite[Theorems II.2.13 and II.6.4]{bg}, and the primitive ideals of $R$ are precisely the prime ideals that are maximal in their $\ch$-strata \cite[Theorem II.8.4]{bg}.


\section{Cauchon's deleting derivations algorithm}
\label{canonicalembedding}
\vskip 2mm


As we have seen in the previous section, the $\ch$-prime ideals of a quantum nilpotent algebra $R$ are key in studying the whole prime 
spectrum. Cauchon's deleting derivations algorithm \cite{c1}, which we summarise below, provides a powerful way of studying the $\ch$-prime 
ideals of $R$.

\subsection{Deleting derivations algorithm} In order to describe the prime spectrum of $R$, Cauchon \cite[Section 3.2]{c1} has
constructed an algorithm called the \emph{deleting derivations algorithm}. This algorithm constructs, for each $j =  N+1, N, \dots, 2 $, an $N$-tuple 
$(x_1^{(j)},\dots,x_N^{(j)})$ of elements of the division ring of fractions $\Fract(R)$ defined as follows:

\begin{enumerate}
\item When $j=N+1$, we set $(x_1^{(N+1)},\dots,x_N^{(N+1)}) :=(x_1,\dots,x_N)$.
\item Assume that $j<N+1$ and that the $x_i^{(j+1)}$ ($i \in \gc 1,N \dc$) are already constructed. 
Then it follows from \cite[Th\'eor\`eme 3.2.1]{c1} that $x_j^{(j+1)} \neq 0$ and, for each $i\in \gc 1,N \dc$, we set
$$x_i^{(j)} :=
\begin{cases}
x_i^{(j+1)} & \quad \mbox{ if }i \geq j \\
\displaystyle{\sum_{k=0}^{+\infty } \frac{(1-q_j)^{-k}}{[k]!_{q_j}} }
\delta_j^k \circ \sigma_j^{-k} (x_i^{(j+1)}) (x_j^{(j+1)})^{-k} & \quad
\mbox{ if }i < j,
\end{cases}$$
where $[k]!_{q_j}=[0]_{q_j} \times \dots \times [k]_{q_j}$ with 
$[0]_{q_j}=1$ and $[i]_{q_j}=1+q_j+\dots+q_j^{i-1}$ when $i \geq 1$. 
\end{enumerate}

For all $j \in \gc 2,N+1 \dc$, we denote by $R^{(j)}$ the subalgebra of $\Fract(R)$ generated by the $x_i^{(j)}$; that is,
$$R^{(j)}:= \mk\langle x_1^{(j)},\dots,x_N^{(j)} \rangle .$$

The following results were proved by Cauchon \cite[Th\'eor\`eme 3.2.1 and Lemme 4.2.1]{c1}. 
For $j \in \gc 2,N+1 \dc$, we have

\begin{enumerate}
\item $R^{(j)}$ is isomorphic to an iterated Ore extension of the form $$\mk[Y_1]\cdots[Y_{j-1};\sigma_{j-1},\delta_{j-1}][Y_j;\tau_j]\cdots[Y_N;\tau_N]$$ 
by an isomorphism that sends $x_i^{(j)}$ to $Y_i$ ($1 \leq i \leq N$), where 
$\tau_j,\dots,\tau_N$ denote the $\mk $-linear automorphisms  such that 
$\tau_{\ell}(Y_i)=\lambda_{\ell,i} Y_i$ ($1 \leq i \leq \ell$).
\item Assume that $j \neq N+1$ and set $S_j:=\{(x_j^{(j+1)})^n ~|~  n\in \mathbb{N} \}=
\{(x_j^{(j)})^n  ~|~  n\in \mathbb{N} \}$.
This is a multiplicative system of regular elements of $R^{(j)}$ and $R^{(j+1)}$, that satisfies the Ore condition
in $R^{(j)}$ and $R^{(j+1)}$. Moreover we have 
\begin{equation}  \label{matchloc}
R^{(j)}S_j^{-1}=R^{(j+1)}S_j^{-1}.
\end{equation}
\end{enumerate}

It follows from these results that $R^{(j)}$ is a noetherian domain, for all $j\in \gc 2,N+1 \dc$.

As in \cite{c1}, we use the following notation.

\begin{notation}
We set $\Rbar:=R^{(2)}$ and $T_i:=x_i^{(2)}$ for all $i\in \gc 1,N \dc$.
\end{notation}

Note that $x_i^{(i+1)} = x_i^{(i)} = \cdots = x_i^{(2)} = T_i$ for $i \in \gc 1,N \dc$. Hence, the structure of $R^{(j)}$ as an iterated Ore extension can be expressed as
\begin{equation}  \label{R(j)}
R^{(j)} = \mk[x_1^{(j)}] \cdots [x_{j-1}^{(j)}; \sigma_{j-1}, \delta_{j-1}] [T_j; \tau_j] \cdots [T_N; \tau_N].
\end{equation}

It follows from \cite[Proposition 3.2.1]{c1} that  $\Rbar$ is a quantum affine space in the indeterminates $T_1,\dots,T_N$, that is, $\Rbar$ is an iterated Ore extension twisted only by automorphisms. It is for this reason that Cauchon used the expression ``effacement des d\'erivations''. More precisely, let $\Lambda=\left( \lambda_{i,j} \right) \in M_N(\mk ^*)$ be the
multiplicatively antisymmetric matrix where the  $\lambda_{j,i}$ with $i<j$ come from the quantum nilpotent algebra structure of $R$ (Definition \ref{CGLdef}). Thus,
$$\lambda_{j,i}= \begin{cases}
1 & \mbox{ if } i=j \\
\lambda_{i,j}^{-1} & \mbox{ if }i> j.
\end{cases}$$
 Then we have 
 \begin{equation}  \label{Rbar.qtorus}
 \Rbar= \mk \langle T_1,\dots,T_N \mid T_iT_j = \lambda_{i,j} T_j T_i \;\; \forall\, i,j \in \gc 1,N \dc \rangle =\co_{\Lambda}(\mk ^N). 
 \end{equation}

\subsection{Canonical embedding} The deleting derivations algorithm was used by Cauchon in order to relate the prime spectrum of a quantum nilpotent algebra $R$ to the prime spectrum of the associated quantum affine space $\Rbar$. More precisely, he  has used this algorithm to construct embeddings 
\begin{equation}
\varphi_j:{\rm Spec}(R^{(j+1)}) \longrightarrow {\rm Spec}(R^{(j)}) \qquad {\rm for ~} j \in \gc 2,N \dc.
\end{equation} 
Recall from \cite[Section 4.3]{c1} that these embeddings are defined as follows.

Let $P \in  {\rm Spec}(R^{(j+1)})$. Then 
$$ 
\varphi_j (P) := \left\{\begin{array}{ll} 
PS_j^{-1} \cap R^{(j)} & \mbox{ if } x_j^{(j+1)} \notin P \\
g_j^{-1} \left( P/\langle x_j^{(j+1)}\rangle \right) & \mbox{ if } x_j^{(j+1)} \in P, \\
\end{array}\right.
$$
where $g_j$ denotes the surjective homomorphism 
\begin{equation}  \label{gjdef}
g_j:R^{(j)}\rightarrow R^{(j+1)}/\langle x_j^{(j+1)}\rangle \quad \text{defined by} \quad g_j(x_i^{(j)}):=x_i^{(j+1)} + \langle x_j^{(j+1)}\rangle \ \ \forall\, i \in \gc 1,N \dc.
\end{equation}
 (For more details see \cite[Lemme 4.3.2]{c1}.)  It was proved in \cite[Proposition 4.3.1]{c1} that $\varphi_j$ induces an inclusion-preserving and -reflecting homeomorphism from the topological space 
 $$\{P \in  {\rm Spec}(R^{(j+1)}) \mid x_j^{(j+1)} \notin P \}$$
  onto 
  $$\{Q \in  {\rm Spec}(R^{(j)}) \mid x_j^{(j)} \notin Q \};$$ 
also, $\varphi_j$ induces an inclusion-preserving and -reflecting homeomorphism from $$\{P \in  {\rm Spec}(R^{(j+1)}) \mid x_j^{(j+1)} \in P \}$$ onto its image under $\varphi_j$. Note however that, in general, $\varphi_j$ is not a homeomorphism from ${\rm Spec}(R^{(j+1)})$ onto its image.

Composing these embeddings, we get an embedding 
\begin{equation} 
\varphi:=\varphi_2 \circ \dots \circ \varphi_N : 
{\rm Spec}(R) \longrightarrow {\rm Spec}(\Rbar), 
\end{equation} 
which is called the  \emph{canonical embedding} from ${\rm Spec}(R)$ into 
${\rm Spec}(\Rbar)$. 
 
\subsection{Cauchon diagrams}  For any subset $w$ of $\{1,\dots,N\}$, let $K_w$ denote the $\ch$-prime ideal of 
 $\widebar{R}$  generated by the $T_i$ with $i\in w$.
 A subset $w\subseteq\{1,\dots,N\}$ is said to be a {\em Cauchon diagram} for $R$ if 
\[
K_w= \langle T_i\mid i\in w\rangle \in \varphi(\hspec(R)),
\]
in which case we denote by $J_w$ the unique $\ch$-prime ideal of $R$ such that
$$
\varphi(J_w) = K_w \,.
$$

A useful way to represent a Cauchon diagram $w$ is as follows. 
Draw $N$ boxes in a row, and  colour the $i$-th box black if and only 
$i\in w$; the remaining boxes are coloured white. For example, if $N=5$ and $w=\{1,2,5\}$ we draw the diagram

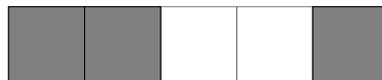
\begin{center}

\begin{tikzpicture}[xscale=1, yscale=1]


\draw[color=gray] (0,2) rectangle (1,3);            
\draw[color=gray] (1,2) rectangle (2,3);            
\draw[color=gray] (2,2) rectangle (3,3);            
\draw[color=gray] (3,2) rectangle (4,3);            
\draw[color=gray] (4,2) rectangle (5,3);            


\draw[fill=gray] (0,2) rectangle (1,3);               
\draw[fill=gray] (1,2) rectangle (2,3);               
\draw[fill=gray] (4,2) rectangle (5,3);               


\end{tikzpicture}
\end{center}
(In fact, the term ``Cauchon diagram'' originates from Cauchon's use of a related representation in the case of quantum matrices.) We write $\#{\rm black}(w)$ and $\#{\rm white}(w)$ for the number of black and white boxes, respectively, in a Cauchon diagram $w$. That is, $\#{\rm black}(w) = |w|$ and $\#{\rm white}(w) = N - |w|$. 

In Section~\ref{building} we will  investigate a way in which we can recolour boxes in a given Cauchon diagram so that the recoloured diagram is still a Cauchon diagram. This will provide us with a way of constructing descending chains of $\ch$-prime ideals of $R$. 

We shall need the fact that the deleting derivations process is $\ch$-equivariant, as we now indicate. The action of $\ch$ on $R$ of course extends to an action of $\ch$ on $\Fract(R)$ by $\mk$-algebra automorphisms (although this action is not rational). Given any $\ch$-eigenvector $v \in \Fract(R)$, denote by $\chi(v)$ the $\ch$-eigenvalue of $v$, so that $h\cdot v = [\chi(v)(h)] v$ for all $h \in \ch$.

\begin{lemma}  \label{chi.xji}
For all $j \in \gc 2,N+1 \dc$ and $i \in \gc 1,N \dc$, the element $x_i^{(j)}$ is an $\ch$-eigenvector with $\chi(x_i^{(j)}) = \chi(x_i)$.
\end{lemma}

\begin{proof} We proceed by induction on $j = N+1,\dots,2$, the case $J=N+1$ holding trivially.

Assume that $j<N+1$ and that the statement holds for all $x_i^{(j+1)}$. The statement for $x_i^{(j)}$ then holds trivially in case $i \ge j$, so assume that $i<j$.

By \cite[Proposition 2.2]{c1}, the map $\theta : \mk\langle x_1^{(j+1)}, \dots, x_{j-1}^{(j+1)} \rangle \rightarrow \Fract(R)$ given by
$$\theta(a) = \displaystyle{\sum_{k=0}^{+\infty } \frac{(1-q_j)^{-k}}{[k]!_{q_j}} }
\delta_j^k \circ \sigma_j^{-k} (a) (x_j^{(j+1)})^{-k}$$
is a $\mk$-algebra homomorphism, and $\theta$ is $\ch$-equivariant by \cite[Lemma 2.6]{llr-ufd}. Therefore $x_i^{(j)} = \theta(x_i^{(j+1)})$ is an $\ch$-eigenvector with $\chi(x_i^{(j)}) = \chi(x_i^{(j+1)}) = \chi(x_i)$, as required.
\end{proof}

For any $j \in \gc 2,N \dc$, the algebra $R^{(j+1)}$ is generated by $\ch$-eigenvectors and its ideal $\langle x_j^{(j+1)} \rangle$ is $\ch$-invariant, so $R^{(j+1)}/\langle x_j^{(j+1)} \rangle$ inherits an induced $\ch$-action. In view of Lemma \ref{chi.xji}, we obtain the following

\begin{corollary}  \label{gjHequi}
For each $j \in \gc 1,N \dc$, the homomorphism $g_j : R^{(j)} \rightarrow R^{(j+1)}/\langle x_j^{(j+1)} \rangle$ of \eqref{gjdef} is $\ch$-equivariant.
\end{corollary}



\section{Gelfand-Kirillov dimension and transcendence degree} 

We denote the Gelfand-Kirillov dimension of a $\mk$-algebra $A$ by $\gk(A)$. A standard reference for results concerning Gelfand-Kirillov 
dimension is \cite{kl}. Three key results that we need are the following. 

\begin{theorem} \label{theorem-gkdim}
Let $R\ = \ \mk[x_1][x_2;\sigma_2,\delta_2]\cdots[x_N;\sigma_N,\delta_N]$ be a quantum nilpotent algebra. Then $\gk(R) = N$. 
\end{theorem} 

\begin{proof} 
This follows easily from \cite[Lemma 2.3]{lr}.
\end{proof} 

\begin{theorem}\label{theorem-gkdrop} 
Let $A$ be a noetherian $\mk$-algebra and let $P$ be a prime ideal of $A$. Then  
\[
\gk(A/P) + \height(P) \le \gk(A).
\] 
\end{theorem} 

\begin{proof}  Noetherianness is more than is needed here -- the result holds if all prime factor rings of $A$ are right Goldie
\cite[Corollary 3.16]{kl}.
\end{proof} 

\begin{proposition}  \label{modfinext}
If $B \subseteq A$ are $\mk$-algebras such that $A$ is finitely generated as a right $B$-module, then
$$\gk (A) = \gk (B).$$
\end{proposition}

\begin{proof} \cite[Proposition 5.5]{kl}.
\end{proof}

We shall also make use of the Gelfand-Kirillov transcendence degree of $\mk$-algebras $A$, denoted $\Tdeg(A)$. See \cite{kl} or \cite{z1}, for instance, for the precise definition. 

\begin{definition} \label{def.tdegstable}
A $\mk$-algebra $A$ is said to be $\Tdeg$-\emph{stable} \cite{z1} if the following hold:
\begin{enumerate}
\item[$\bullet$] $\gk(A)=\Tdeg(A)$. 
\item[$\bullet$] For every multiplicative system $S$ of regular elements  of
$A$ that satisfies the Ore condition, we have:  $\Tdeg(S^{-1}A)=\Tdeg(A)$.
\end{enumerate}
\end{definition}

A key instance of this property is

\begin{lemma}  \label{qtorTdegstable}
Every quantum torus is Tdeg-stable.
\end{lemma}

\begin{proof} 
\cite[Corollary 2.2]{Lor} or \cite[Proposition 7.2]{z1}.
\end{proof}

We excerpt the following key result from \cite{z1}.

\begin{proposition}  \label{transfer.Tdeg}
{\rm \cite[Proposition 3.5(4)]{z1}}
Let $A$ be a semiprime Goldie $\mk$-algebra and let $B \subseteq A$ be a semiprime Goldie subalgebra such that $\Fract B = \Fract A$. If $A$ is $\Tdeg$-stable, then $\gk(B) = \gk(A)$ and $B$ is $\Tdeg$-stable.
\end{proposition}

An immediate application is that any quantum nilpotent algebra $R$ is $\Tdeg$-stable, since, by \cite[Theorem 4.6]{gy0}, $R$ is trapped between a quantum affine space $\co_\bfq(\mk^N)$ and the corresponding quantum torus $\co_\bfq((\mk^*)^N)$. In fact, $\Tdeg$-stability holds for all $\ch$-prime factors of $R$, by the same argument used in \cite[Proposition 1.3.2.2]{stephane} for the case of quantum matrices. This relies on a result of Cauchon \cite[Th\'eor\`eme 5.4.1]{c1} which shows that for any Cauchon diagram $w$ of $R$, the algebras $R/J_w$ and $\Rbar/K_w$ have isomorphic localizations. For later use, we require the following $\ch$-equivariant version of the result. 

Recall that the algebra $\Rbar/K_w$ is a quantum affine space with canonical generators given by the cosets of those $T_i$ with $i \in \gc 1,N \dc \setminus w$, where $N$ is the length of $R$. Let $\ce_w$ denote the multiplicative set in $\Rbar/K_w$ generated by $\{ T_i \mid i \in \gc 1,N \dc \setminus w \}$. Observe that $\ce_w$ is an Ore set in $\Rbar/K_w$ and that $(\Rbar/K_w)\ce_w^{-1}$ is a quantum torus of rank $N-|w|$.

\begin{theorem}  \label{matchloc2}
Let $R$ be a quantum nilpotent algebra of length $N$ and let $w$ be a Cauchon diagram for $R$. There exists an Ore set $\cf_w$ of regular $\ch$-eigenvectors in $R/J_w$ such that

{\rm(a)} There is an $\ch$-equivariant $\mk$-algebra isomorphism $(R/J_w)\cf_w^{-1} \rightarrow (\Rbar/K_w)\ce_w^{-1}$.

{\rm(b)} $\spec_{J_w}(R) = \{ P \in \spec(R) \mid P\supseteq J_w \ \text{and} \ (P/J_w) \cap \cf_w = \varnothing \}$.
\end{theorem}

\begin{proof} The Ore set we label $\cf_w$ is denoted $\Sigma_{N+1}$ in \cite[Subsection 5]{c1}, where we take $P=J_w$. It is obtained as the end result of a sequence of Ore sets $\Sigma_2, \dots, \Sigma_{N+1}$ in subalgebras $A_2\cong \Rbar/K_w,\dots,A_{N+1} = R/J_w$ of $\Fract(R/J_w)$, where $\Sigma_2$ is the image of $\ce_w$ in $A_2$. Specifically, $\Sigma_{j+1} = \Sigma_j \cap A_{j+1}$ for $j \in \gc2,N\dc$.

The action of $\ch$ on $R/J_w$ by automorphisms extends to an action on $\Fract(R/J_w)$ by automorphisms (although no longer rational).  As one notes, the elements of $\Sigma_2$ are regular $\ch$-eigenvectors, so the same holds for $\Sigma_{N+1} = \cf_w$. That $\Sigma_{N+1}$ satisfies the Ore condition in $A_{N+1}$ is proved in \cite[Proposition 5.4.4(2)]{c1}. 

(a) This isomorphism corresponds to an equality in \cite{c1}, due to identifications made in that paper.

In \cite{c1}, $\Rbar/K_w$ is identified with $A_2$ via a $\mk$-algebra epimorphism $f_2 : R_2 = \Rbar \rightarrow A_2$ with kernel $K_w$ \cite[Proposition 5.4.1(2)]{c1}. It follows from Lemma \ref{chi.xji} that $f_2$ is $\ch$-equivariant. Therefore the induced isomorphism $(\Rbar/K_w)\ce_w^{-1} \rightarrow A_2 \Sigma_2^{-1}$ is $\ch$-equivariant. Since $A_2 \Sigma_2^{-1} = A_{N+1} \Sigma_{N+1}^{-1} = (R/J_w) \cf_w^{-1}$ \cite[Proposition 5.4.4(3)]{c1}, we obtain (the inverse of) the desired $\ch$-equivariant isomorphism.

(b) If $P \in \spec_{J_w}(R)$, then $P/J_w$ contains no nonzero $\ch$-invariant ideals, so it cannot contain any $\ch$-eigenvectors. Hence, $(P/J_w) \cap \cf_w = \varnothing$. 

Conversely, suppose $P \in \spec(R)$ with $P \supseteq J_w$ and $(P/J_w) \cap \cf_w = \varnothing$. Then $I = (P:\ch)$ is a proper $\ch$-invariant ideal of $R$ such that $I \supseteq J_w$ and $(I/J_w) \cap \cf_w = \varnothing$, and so $(I/J_w)\cf_w^{-1}$ is a proper $\ch$-invariant ideal of $(R/J_w)\cf_w^{-1}$. Under the isomorphism in (a), $(I/J_w)\cf_w^{-1}$ corresponds to a proper $\ch$-invariant ideal $I'$ of $(\Rbar/K_w)\ce_w^{-1}$. However, by \cite[Lemme 5.5.3]{c1}, $(\Rbar/K_w)\ce_w^{-1}$ is an $\ch$-simple ring, i.e., it has no nonzero proper $\ch$-invariant ideals. Consequently, $I' = 0$, whence $(I/J_w)\cf_w^{-1} = 0$. This forces $I/J_w = 0$, that is, $(P:\ch) = I = J_w$. Therefore $P \in \spec_{J_w}(R)$.
\end{proof}

\begin{corollary}  \label{GKRmodJw}
If $R$ is a quantum nilpotent algebra of length $N$ and $w$ is a Cauchon diagram for $R$, then $R/J_w$ is $\Tdeg$-stable and $\gk(R/J_w) = N-|w| = \#{\rm white}(w)$.
\end{corollary}

\begin{proof} By Lemma \ref{qtorTdegstable}, the quantum torus $(\Rbar/K_w)\ce_w^{-1}$ is $\Tdeg$-stable, and we observe that this algebra has GK-dimension $N-|w|$. In view of Theorem \ref{matchloc2}, the algebra $(R/J_w)\cf_w^{-1}$ has the same properties. Since $\Fract R/J_w = \Fract (R/J_w)\cf_w^{-1}$, Proposition \ref{transfer.Tdeg} yields the desired conclusions.
\end{proof}



\section{Building height of $\ch$-primes by using black boxes} \label{building}

Let $R\ = \ \mk[x_1][x_2;\sigma_2,\delta_2]\cdots[x_N;\sigma_N,\delta_N]$ be a quantum nilpotent algebra of length $N$. In Theorem \ref{theorem-gkdim} 
we have seen that  $\gk(R)=N$, and Theorem~\ref{theorem-gkdrop} says
that 
\[
\gk(R/P)+\height(P) \leq \gk(R) 
\] 
for each prime ideal $P$ of $R$. 
 We aim to prove that these inequalities are actually equalities.

In this section, we establish the height formula for $\ch$-prime ideals of $R$. As an abbreviation, we write $P\hp A$ to denote that $P$ is an $\ch$-prime ideal 
in $A$ for any ring $A$ equipped with a group action by $\ch$.

We will show that if we take the Cauchon diagram for an $\ch$-prime ideal $P$ of $R$ 
 and change the final black box to a white box then we again get the Cauchon diagram for an $\ch$-prime ideal and that 
 this new $\ch$-prime ideal 
 is necessarily contained in the original $\ch$-prime ideal. In this way, we can build a descending chain of $\ch$-prime ideals with length equal to the number of black boxes. This provides a lower bound for the height of $P$. Combining this with Corollary \ref{GKRmodJw} and Theorem \ref{theorem-gkdim}, we see that 
 Tauvel's height formula for the $\ch$-primes of $R$ then follows easily, and that the height of an $\ch$-prime $P$ of $R$ is equal to the 
 number of black boxes in the  Cauchon diagram for $P$.

\subsection{Black box removal}
Fix an $\ch$-prime ideal  $J_w$ of $R$, with $w$ nonempty. Let $k$ be the maximal member of $w$.  
We want to show that $w':= w\backslash\{k\}$ is a Cauchon diagram, so that there is an 
$\ch$-prime ideal $J_{w'}$ in $R$ with $\varphi(J_{w'})= K_{w'}$. In order to do this, we need to reverse the procedure described in Section~\ref{canonicalembedding}. We keep the notation of that section; in particular, $\widebar{R}$ denotes the quantum affine space that is reached at the end of the deleting derivations process, and $\varphi:\spec(R)\hookrightarrow \spec(\widebar{R})$ is the canonical embedding.

Recall the iterated Ore extension presentation of $R^{(j)}$ from \eqref{R(j)}. 
When $j\leq k$, because of the nature of the automorphisms $\tau_i$ with 
$i\geq j$, we may write (with a slight abuse of notation)
\[
R^{(j)} = A^{(j)}[T_k;\tau_k],
\]
where 
\[
A^{(j)}:= 
\mk [x_1^{(j)}]\cdots[x_{j-1}^{(j)};\sigma_{j-1},\delta_{j-1}][T_j;\tau_j]\cdots
[T_{k-1};\tau_{k-1}][T_{k+1};\tau_{k+1}]\cdots
[T_N;\tau_N].
\]
Here for $i>k$ we have written $\tau_i$ for the restriction of the original $\tau_i$ to the algebra 
$$
\mk \langle x_1^{(j)}, \dots, x_{j-1}^{(j)}, T_j, \dots, T_{k-1}, T_{k+1}, \dots, T_{i-1} \rangle.
$$

The following technical lemma, needed in the next result, gives a sufficient 
criterion for recognising when an $\ch$-ideal is induced from the base ring in an Ore extension endowed with a suitable $\ch$-action. 

\begin{lemma} \label{lemma-induced}
Let $B=A[X;\sigma]$ be an Ore extension of $\mk$-algebras, and assume that $\sigma$ extends to an automorphism $\sighat$ of $B$ such that $\sighat(X) = q X$ for some $q\in \mk^*$ which is not a root of unity. Let $I$ be a $\sighat$-invariant ideal 
of $B$, and suppose that $aX^n \in I$ implies $a\in I$, for any $a\in A$ and $n \in \mn$. Then
\begin{enumerate}
\item[\rm(i)] $I= (I\cap A)B$.
\item[\rm(ii)] The natural map $(A/(I\cap A))[\widebar{X};\widebar{\sigma}] \rightarrow B/I$ is an isomorphism.
\end{enumerate}
\end{lemma}  

\begin{proof} (i) We adapt the proof of \cite[Lemma 2.2]{llr-ufd}. If (i) fails, there exists an element $b = \sum_i b_iX^i \in I$ with all $b_i \in A$ but some $b_j \notin I$. Set $m := \min \{ i \mid b_i \notin I \}$ and $n := \max \{ i \mid b_i \notin I \}$. Since $b_i \in I$ for $i \notin \gc m,n \dc$, we may remove terms with these indices from $b$, that is, there is no loss of generality in assuming $b = \sum_{i=m}^n b_i X^i$. 

If $m=n$, then $b_n X^n \in I$ with $b_n \notin I$, contradicting our hypotheses. Thus, $m<n$. Without loss of generality, $m-n$ is minimal among instances of elements with the properties of $b$. 

Now $I$ contains the elements
$$
Xb = \sum_{i=m}^n \sigma(b_i) X^{i+1} \qquad \text{and} \qquad \sighat(b) = \sum_{i=m}^n q^i \sigma(b_i) X^i,
$$
and so it also contains the element
$$
\sighat(b) X - q^n X b = \sum_{i=m}^{n-1} (q^i - q^n) \sigma(b_i) X^{i+1}.
$$
The minimality of $m-n$, together with the assumption that $q$ is not a root of unity, implies that $\sigma(b_i) \in I$ for all $i \in \gc m,n-1 \dc$. But then $b_m = \sighat^{-1}(\sigma(b_m)) \in I$, a contradiction.

(ii) This follows easily from (i), e.g. by using the comment in \cite[2.1(vi)]{gl1}.
\end{proof}



\begin{proposition} \label{largerCauchon}
 Let $w$ be a nonempty Cauchon diagram for $R$ and let $k := \max w$. The set  $w':=w\backslash\{k\}$ is a Cauchon diagram with $J_{w'}
\subsetneqq J_w$.
\end{proposition} 

\begin{proof} 
If $w = \{k\}$, then $w' = \varnothing$. Clearly $\varphi(0) = 0 = K_\varnothing \in \Spec(\Rbar)$, so $w'$ is a Cauchon diagram and $J_{w'} = 0 \subsetneqq J_w$.
Hence, we may assume that $w \supsetneqq \{k\}$; in particular, $k\geq 2$. Note that if $s\neq k$ then $s\in w'$ if and only if $s\in w$.

As $w$ is a Cauchon diagram, there exist $\ch$-prime ideals $J_w^{(i)}\hp R^{(i)}$ 
for $i \in \gc 2,N+1 \dc$ such that $\varphi_i(J_w^{(i+1)})=J_w^{(i)}$ for $i \in \gc 2,N \dc$, where $J_w^{(N+1)} = J_w$ and $J_w^{(2)} = K_w$.  We aim to construct a corresponding sequence of $\ch$-prime ideals $J_{w'}^{(i)} \hp R^{(i)}$.

First, we show that there are $\ch$-prime ideals 
$J_{w'}^{(2)}\hp R^{(2)},\dots,J_{w'}^{(k)}\hp R^{(k)}$ with 
$A^{(i)}\cap J_{w}^{(i)}\subseteq J_{w'}^{(i)}$ for each $i \in \gc 2,k \dc$ 
and 
$\varphi_i(J_{w'}^{(i+1)})=J_{w'}^{(i)}$ for each $i \in \gc 2,k-1 \dc$.
To begin, set $J_{w'}^{(2)} :=K_{w'}$
and observe that $A^{(2)}\cap J_w^{(2)} = \langle T_j \mid j\in w' \rangle_{A^{(2)}}\subseteq J_{w'}^{(2)}$.

Now assume that $2\leq s<k$ and that $\ch$-prime ideals $J_{w'}^{(2)}\hp R^{(2)},\dots,J_{w'}^{(s)}\hp R^{(s)}$ have been defined with 
$A^{(i)}\cap J_{w}^{(i)}\subseteq J_{w'}^{(i)}$ for each 
$i \in \gc 2,s \dc$ and 
$\varphi_i(J_{w'}^{(i+1)})=J_{w'}^{(i)}$  for each $i \in \gc 2,s-1 \dc$. We need to find an $\ch$-prime ideal $J_{w'}^{(s+1)}$ in $R^{(s+1)}$ such that 
$A^{(s+1)}\cap J_{w}^{(s+1)}\subseteq J_{w'}^{(s+1)}$ and 
$\varphi_s(J_{w'}^{(s+1)})=J_{w'}^{(s)}$. We distinguish between two possible cases, depending on whether or not $s \in w'$. 

\noindent{\bf Case (i)} Assume that  $s\not\in w'$.
 
As $s<k$, we know that $s\not\in w$. 
In this case, set 
$J_{w'}^{(s+1)}:= J_{w'}^{(s)}[T_s^{-1}]\cap R^{(s+1)}$. It is easy to see that 
 $J_{w'}^{(s+1)}\hp R^{(s+1)}$ (observe that $T_s \notin J^{(2)}_{w'}$ implies $T_s \notin J^{(s)}_{w'}$) and we can check (using \eqref{matchloc}) that 
$\varphi_s(J_{w'}^{(s+1)})=J_{w'}^{(s)}$. Now, 
\begin{align*}
A^{(s+1)}\cap J_{w}^{(s+1)}
&= A^{(s+1)}\cap \bigl( J_{w}^{(s)}[T_s^{-1}]\cap R^{(s+1)} \bigr) \subseteq A^{(s)}[T_{s}^{-1}]\cap J_{w}^{(s)}[T_s^{-1}]\cap R^{(s+1)} \\
&\subseteq (A^{(s)}\cap J_{w}^{(s)})[T_s^{-1}]\cap R^{(s+1)} \subseteq J_{w'}^{(s)}[T_s^{-1}]\cap R^{(s+1)} = J_{w'}^{(s+1)},
\end{align*} 
as required to finish Case (i). (Here, the first containment and the last equality follow from \eqref{matchloc}, and the last 
containment is given by the inductive hypothesis.)

\noindent{\bf Case (ii)} Assume that  $s\in w'$.
In this case, $s\in w$. 

We have $T_s \in J^{(2)}_{w'}$ and $T_s \in J^{(2)}_w$. It follows that $T_s \in J^{(s)}_{w'}$ as well as $T_s \in J^{(s)}_w$ and $T_s \in J^{(s+1)}_w$.

In order to set up an application of Lemma \ref{lemma-induced}, let
\begin{align*}
A^{(s+1)}_{k-1} &:= \mk \langle x_1^{(s+1)}, \dots, x_{s}^{(s+1)}, T_{s+1}, \dots, T_{k-1} \rangle \subset A^{(s+1)} \\
 R^{(s+1)}_k &:= \mk \langle x_1^{(s+1)}, \dots, x_s^{(s+1)}, T_{s+1}, \dots, T_k \rangle \subset R^{(s+1)},
\end{align*}
so that $R^{(s+1)}_k = A^{(s+1)}_{k-1}[T_k; \tau_k]$. In view of Lemma \ref{chi.xji}, the restriction of $(h_k\cdot)$ to $R^{(s+1)}_k$ (where $h_k$ is the element of $\ch$ occurring in part (iii) of Definition \ref{CGLdef}) yields an automorphism $\tauhat_k$ which extends $\tau_k$ and satisfies $\tauhat_k(T_k) = q_k T_k$.

\noindent{\bf Claim} Suppose that $a\in A^{(s+1)}_{k-1}$ and 
$aT_k^n \in \langle T_s\rangle_{R^{(s+1)}_k}$ for some $n \ge 0$. Then 
$a\in  \langle T_s\rangle_{A^{(s+1)}_{k-1}}$.

\noindent{\em Proof of Claim} Recall from \eqref{Rbar.qtorus} that 
$T_kT_s=\lambda T_sT_k$ where $\lambda = \lambda_{k,s} \in \mk^*$. 

Suppose that $aT_k^n =\sum_i\, c_iT_sd_i$, with $c_i,d_i\in R^{(s+1)}_k$. Write $c_i=\sum_{\alpha}\, c_{i\alpha}T_k^\alpha$ and 
$d_i =\sum_\beta\,d_{i\beta}T_k^\beta$ with $c_{i\alpha}, d_{i\beta}\in A^{(s+1)}_{k-1}$. Thus, 
\begin{align*} 
aT_k^n &= \sum_{i,\alpha,\beta}\, (c_{i\alpha}T_k^\alpha )T_s(d_{i\beta}T_k^\beta) = \sum_{i,\alpha,\beta}\, c_{i\alpha}\lambda^\alpha T_sT_k^\alpha d_{i\beta}T_k^\beta = \sum_{i,\alpha,\beta}\, c_{i\alpha}\lambda^\alpha T_s\tau_k^\alpha(d_{i\beta})T_k^\alpha T_k^\beta  \\
&= \sum_{i,\alpha,\beta}\, c_{i\alpha}\lambda^\alpha T_s\tau_k^\alpha(d_{i\beta})T_k^{\alpha+\beta} = \sum_{m=0}^\infty\, \biggl(\, \sum_i\, \sum_{\alpha+\beta = m}\, c_{i\alpha}\lambda^\alpha T_s\tau_k^\alpha(d_{i\beta}) \biggr) T_k^m.
\end{align*} 
Since $R^{(s+1)}_k =A^{(s+1)}_{k-1}[T_k;\tau_k]$ is an Ore extension, it follows that 
$$
a= \sum_i\, \sum_{\alpha=0}^n c_{i\alpha}\lambda^\alpha T_s\tau_k^\alpha(d_{i, n-\alpha}) \in \langle T_s\rangle_{A^{(s+1)}_{k-1}} \,,
$$
as required to establish the truth of the claim. 

The case $n=0$ of the claim implies that $\langle T_s\rangle_{R^{(s+1)}_k} \cap A^{(s+1)}_{k-1} = \langle T_s\rangle_{A^{(s+1)}_{k-1}}$.
Applying Lemma \ref{lemma-induced}, we obtain
$$
\langle T_s\rangle_{R^{(s+1)}_k} = \langle T_s\rangle_{A^{(s+1)}_{k-1}} R^{(s+1)}_k \,.
$$
Since the ideals $\langle T_s\rangle_{R^{(s+1)}_k}$ and $\langle T_s\rangle_{A^{(s+1)}_{k-1}}$ are invariant under $\tau_{k+1}, \dots, \tau_N$, it follows that
$$
\langle T_s\rangle_{R^{(s+1)}} = \langle T_s\rangle_{R^{(s+1)}_k} R^{(s+1)} \qquad \text{and} \qquad
\langle T_s\rangle_{A^{(s+1)}} = \langle T_s\rangle_{A^{(s+1)}_{k-1}} A^{(s+1)},
$$
whence $\langle T_s\rangle_{R^{(s+1)}} = \langle T_s\rangle_{A^{(s+1)}} R^{(s+1)}$.
Consequently,
\begin{equation}  \label{R/T.isom}
\text{the natural map} \;\; \left(
A^{(s+1)}/\langle T_s\rangle_{A^{(s+1)}}\right)
[\widebar{T}_k;\widebar{\sigma}_k] \longrightarrow R^{(s+1)}/\langle T_s\rangle \;\; \text{is an isomorphism}.
\end{equation}

Consider the map $g_s:R^{(s)}\longrightarrow R^{(s+1)}/\langle T_s\rangle$
that arises in the deleting derivations process \eqref{gjdef}. This map  induces an 
isomomorphism from 
$R^{(s)}/ \ker(g_s)$ to $R^{(s+1)}/ \langle T_s\rangle$, 
and we know that 
$\ker(g_s)\subseteq J_{w}^{(s)} = g_s^{-1}(J_{w}^{(s+1)}/\langle T_s\rangle)$. In fact, there is the isomorphism 
\[
R^{(s)}/J_{w}^{(s)}\cong R^{(s+1)}/J_{w}^{(s+1)}
\]
that is induced by $g_s$. 

We shall prove that $\ker(g_s)\subseteq J_{w'}^{(s)}$. Let $x\in\ker(g_s)$, and write $x=\sum_i a_i^{(s)}T_k^i$ with each 
$a_i^{(s)}\in A^{(s)}$. Now by \eqref{R/T.isom}, $g_s(x)=0$ implies $g_s(a_i^{(s)}) = 0$, for each $i$; so $a_i^{(s)}\in\ker(g_s)\cap A^{(s)}
\subseteq J_w^{(s)}\cap A^{(s)} \subseteq J_{w'}^{(s)}$, where the final 
containment is given by the inductive hypothesis. Hence, 
$\ker(g_s)\subseteq J_{w'}^{(s)}$, as required. 

Since $g_s$ is $\ch$-equivariant (Corollary \ref{gjHequi}), it follows that there is an $\ch$-prime ideal $J_{w'}^{(s+1)}\hp R^{(s+1)}$ such 
that $J^{(s+1)}_{w'} \supseteq \langle T_s\rangle$ and $g_s$ induces an isomorphism 
\[
R^{(s)}/J_{w'}^{(s)}\cong R^{(s+1)}/J_{w'}^{(s+1)}.\\[2ex]
\]
In particular, $J^{(s)}_{w'} = g_s^{-1}(J^{(s+1)}_{w'}/\langle T_s\rangle) = \varphi_s(J^{(s+1)}_{w'})$.

Let $z\in A^{(s+1)}\cap J_w^{(s+1)}$. There exists $z'\in A^{(s)}\cap J_w^{(s)}$ 
such that 
$g_s(z')=\widebar{z}\in R^{(s+1)}/\langle T_s\rangle$. 
Now, $z'\in J_{w'}^{(s)}$, by the inductive hypothesis; 
so $\widebar{z}=g_s(z')\in g_s(J_{w'}^{(s)}) =  J_{w'}^{(s+1)}/\langle T_s\rangle$. 
Since $T_s\in J_{w'}^{(s+1)}$, this establishes the required 
inclusion $A^{(s+1)}\cap J_{w}^{(s+1)}\subseteq J_{w'}^{(s+1)}$ in this case.
 This finishes Case (ii).
 
 At this stage, we have constructed $\ch$-prime ideals 
$J_{w'}^{(i)}\hp R^{(i)}$ for $i \in \gc 2,k \dc$ with 
$\varphi_i(J_{w'}^{(i+1)})=J_{w'}^{(i)}$ for $i= \gc 2,k-1 \dc$. We still need to construct $\ch$-prime ideals $J_{w'}^{(i)}\hp R^{(i)}$ for $i \in \gc k+1,N+1 \dc$ with 
$\varphi_i(J_{w'}^{(i+1)})=J_{w'}^{(i)}$ for $i \in \gc k,N \dc$. However, for $s\geq k$, we 
know that $s\not\in w'$; so the same reasoning as in the first part of Case (i) 
above does what is required.

Now $J_{w'}^{(N+1)}$ is an $\ch$-prime ideal of $R$ with $\varphi(J_{w'}^{(N+1)}) = J^{(2)}_{w'} = K_{w'}$, whence $w'$ is the Cauchon diagram for $J_{w'}^{(N+1)}$ and $J_{w'}^{(N+1)} = J_{w'}$. The fact that $J_{w'}\subsetneqq J_w$ follows from \cite[Theorem 1.4]{bl}
\end{proof}

\begin{corollary}  \label{htHprime}
Let $J$ be an $\ch$-prime ideal in the quantum nilpotent algebra $R$. If $w$ is the Cauchon diagram of $J$, then $\height(J)\geq \#{\rm black}(w)$. 
\end{corollary} 

\subsection{Height formula for $\ch$-primes}

\begin{theorem} \label{Tauvel-Hprimes}
Let $J$ be an $\ch$-prime ideal of the quantum nilpotent algebra $R$ with Cauchon diagram $w$. Then 
\[
\gk(R/J) + \height(J)=\gk(R); 
\]
that is, Tauvel's height formula holds for the $\ch$-prime ideals of any 
quantum nilpotent algebra. Furthermore, $\gk(R/J)=\#{\rm white}(w)$ and 
$\height(J)=\#{\rm black}(w)$. 
\end{theorem} 

\begin{proof}  We already have $\gk(R/J)=\#{\rm white}(w)$, by Corollary \ref{GKRmodJw}. By using Corollary \ref{htHprime} and Theorem \ref{theorem-gkdrop}, we see that 
\[
N=\gk(R)\geq \gk(R/J)+\height(J)\geq \#{\rm white}(w)+\#{\rm black}(w)= N.
\]
Tauvel's height formula for $J$ follows, as does the claim about $\height(J)$.
\end{proof}

In order to extend this result to arbitrary prime ideals, we need to further employ Tdeg-stability. The necessary details are given in the next section. 

The proof of Corollary \ref{htHprime} shows that for any $\ch$-prime ideal $J$ of $R$ with Cauchon diagram $w$, there is a strictly descending chain $J_0 = J \supsetneq J_1 \supsetneq \cdots \supsetneq J_m = 0$ of $\ch$-primes of $R$ with $m = \#{\rm black}(w)$. In other words, the height of $J$ within the poset $\hspec R$ is at least $\#{\rm black}(w)$. Since this value, which we denote $\height_{\hspec R}(J)$, is dominated by $\height_{\spec R}(J) := \height(J)$, we obtain the following from Theorem \ref{Tauvel-Hprimes}.

\begin{corollary}  \label{Hht=ht}
If $J$ is any $\ch$-prime ideal of the quantum nilpotent algebra $R$, then 
\begin{equation}  \label{HhtJ=htJ}
\height_{\hspec R}(J) = \height_{\spec R}(J).
\end{equation}
\end{corollary}

Equation \eqref{HhtJ=htJ} had previously been established only under the hypothesis that $\hspec R$ has \emph{$\ch$-normal separation}, meaning that for any $\ch$-prime ideals $J \subsetneq K$ of $R$, there is an $\ch$-eigenvector $u \in K \setminus J$ such that $u+J$ is normal in $R/J$ and the corresponding automorphism of $R/J$ is given by some element of $\ch$ (see \cite[Proposition 5.9]{y1}). As a consequence, \eqref{HhtJ=htJ} was known for quantum nilpotent algebras of the form $U^w_-(\gfrak)$ (i.e., \emph{quantum Schubert cell algebras}) \cite[Proof of Theorem 5.8]{y1} and for cocycle twists of the $U^w_-(\gfrak)$ \cite[Section 5]{y2}.


\section{T-degree stability for primitive quotients of quantum nil\-potent algebras}

We now require some more precise information about $\Tdeg$ and $\Tdeg$-stability. 
This will be obtained using the \emph{lower transcendence degree} (over $\mk$) of a $\mk$-algebra $A$, as defined in \cite{z2}. This degree, denoted $\Ld(A)$, is a value in $\mr_{\ge0} \cup \{\infty\}$; we refer to \cite{z2} for the definition. (We do not require lower transcendence degrees over division subalgebras of $A$.)

\begin{lemma}  \label{finiteextLd}
Let $B \subseteq A$ be prime Goldie $\mk$-algebras such that all regular elements of $B$ are also regular in $A$. If $A$ is finitely generated as a right $B$-module, then $\Ld(A) = \Ld(B)$.
\end{lemma}

\begin{proof} By \cite[Theorem 0.3(2)]{z2}, $\Ld(A) = \Ld(\Fract(A))$ and $\Ld(B) = \Ld(\Fract(B))$, so it remains to show that $\Ld(\Fract(A)) = \Ld(\Fract(B))$. This will follow from \cite[Theorem 0.3(1)]{z2} once we show that $\Fract(A)$ is finitely generated as a right $\Fract(B)$-module, since $\Fract(B)$ is artinian.

Due to the assumption on regular elements, we can identify $\Fract(B)$ with a subalgebra of $\Fract(A)$. We have $A = \sum_{i=1}^n a_iB$ for some $a_i \in A$. Set $D := A\cdot\Fract(B) = \sum_{i=1}^n a_i \Fract(B)$, a right $\Fract(B)$-submodule of $\Fract(A)$ which is finitely generated and thus artinian. Any regular element $a\in A$ is invertible in $\Fract(A)$, whence $aD \cong D$. Since $aD \subseteq D$ and $D$ is artinian on the right, $aD = D$, whence $a^{-1}D = D$. Now $D$ is a left ideal of $\Fract(A)$, and it contains $A$, so $D = \Fract(A)$. Therefore $\Fract(A)$ is a finitely generated right $\Fract(B)$-module, as desired.
\end{proof}

\begin{proposition}  \label{qtorusmodmax}
If $T$ is a quantum torus over $\mk$ and $M$ is a maximal ideal of $T$, then $T/M$ is $\Tdeg$-stable.
\end{proposition}

\begin{proof} As noted in \cite[pp. 159--60]{z2}, it suffices to show that $T/M$ is \emph{$\Ld$-stable} in the sense that $\Ld(T/M) = \gk(T/M)$.

Extension and contraction give inverse bijections between the sets of ideals in $Z(T)$ and $T$ (e.g., \cite[Proposition II.3.8]{bg}), so $\mfrak := M \cap Z(T)$ is a maximal ideal of $Z(T)$ and $M = \mfrak T$. Since $Z(T)$ is a Laurent polynomial algebra over $\mk$ (e.g., \cite[Lemma II.3.7(e)]{bg}), the field $Z(T)/\mfrak$ is finite dimensional over $\mk$.

Write $T = \Obfq((\mk^*)^n)$ for some $n \in \mn$ and some multiplicatively skewsymmetric matrix $\bfq = (q_{ij}) \in M_n(\mk^*)$. Let $y_1^{\pm1}, \dots, y_n^{\pm1}$ be a standard set of generators for $T$, so that
$$ T = \mk \langle\, y_1^{\pm1}, \dots, y_n^{\pm1} \mid y_i y_j = q_{ij} y_j y_i \ \forall\; i,j \in \gc 1,n \dc\, \rangle.$$
Then let $\{ y^a \mid a \in \mz^n \}$ be the corresponding $\mk$-basis for $T$, where
$$
y^a := y_1^{a_1} y_2^{a_2} \cdots y_n^{a_n} \qquad \forall\; a = (a_1,\dots,a_n) \in \mz^n.
$$
Set $Z := \{ a \in \mz^n \mid y^a \in Z(T) \}$, so that $Z(T) = \bigoplus_{a\in Z} \mk y^a$ (e.g., \cite[Lemma II.3.7(a)]{bg}).

Choose a subgroup $W \subseteq \mz^n$ maximal with respect to the property $W \cap Z = 0$. Then $\mz^n/(Z \oplus W)$ is finite, and $Z \oplus W$ is free abelian of rank $n$. Choose bases $(b_1,\dots,b_l)$ and $(b_{l+1},\dots,b_n)$ for $Z$ and $W$, respectively, and set $z_i := y^{b_i}$ for all $i$. Then
$$
Z(T) = \bigoplus_{a\in Z} \mk y^a = \mk[z_1^{\pm1}, \dots, z_l^{\pm1}].
$$
Next, set
$$
C := \bigoplus_{a\in W} \mk y^a = \mk\langle z_{l+1}^{\pm1}, \dots, z_n^{\pm1} \rangle,$$
a quantum torus over $\mk$ of rank $r := n-l$. Finally, set 
$$
B := \bigoplus_{a\in Z\oplus W} \mk y^a = \mk\langle z_1^{\pm1}, \dots, z_n^{\pm1} \rangle,$$
a $\mk$-subalgebra of $T$. Observe that the multiplication map $Z(T) \otimes_\mk C \rightarrow T$ gives a $\mk$-algebra isomorphism of $Z(T) \otimes_\mk C$ onto $B$. We identify $B$ with $Z(T) \otimes_\mk C$ via this isomorphism.

We next show that $C$ is a central simple $\mk$-algebra (meaning only that $C$ is a simple ring with center $\mk$). Simplicity will follow from \cite[Proposition 1.3]{McPe} once we show that $Z(C) = \mk$. We know that $Z(C)$ is spanned by those $y^a$ with $a\in W$ and $y^a \in Z(C)$. Let $s := |\mz^n/(Z \oplus W)|$, so that $sb \in Z \oplus W$ for all $b\in \mz^n$. If $a \in W$ and $y^a \in Z(C)$, then $y^a$ commutes with $y^{sb}$ for any $b \in \mz^n$. There is a scalar $\lambda_{a,b} \in \mk^*$ such that $y^a y^b = \lambda_{a,b} y^b y^a$, whence $y^a y^{sb} = \lambda_{a,b}^s y^{sb} y^a = \lambda_{a,b}^s y^a y^{sb}$. But then $\lambda_{a,b}^s = 1$ and so $y^{sa} y^b = y^b y^{sa}$. Consequently, $y^{sa} \in Z(T)$ and $sa \in Z$. Since also $a\in W$, we must have $a=0$. Therefore $Z(C) = \mk$, as required.

Since $C$ is a central simple $\mk$-algebra, $\mfrak B = \mfrak \otimes_\mk C$ is a maximal ideal of $B$. Now $B/\mfrak B = (Z(T)/\mfrak) \otimes_\mk C$, and we identify $C$ with a subalgebra of $B/\mfrak B$. Note that any $\mk$-basis for $Z(T)/\mfrak$  provides a finite basis for $B/\mfrak B$ as a free right and left $C$-module. In particular, it follows that all regular elements of $C$ are also regular in $B/\mfrak B$. Thus, Lemma \ref{finiteextLd} implies $\Ld(B/\mfrak B) = \Ld(C)$. Now $\gk(C) = r$, and $C$ is $\Ld$-stable by \cite[Corollary 6.3(1)]{z2}, so $\Ld(C) = r$. Taking account of Proposition \ref{modfinext}, we therefore have
\begin{equation}  \label{LdB/mB}
\Ld(B/\mfrak B) = \gk(B/\mfrak B) = r.
\end{equation}

Since $\mfrak B$ is a maximal ideal of $B$, it must equal $B\cap M$. We then identify $B/ \mfrak B$ with its image in $T/M$. Observe that $T$ is a free right and left $B$-module with a basis $\{ y^{u_1}, \dots, y^{u_s} \}$, where $\{ u_1,\dots,u_s \}$ is a complete set of coset representatives for $\mz^n/(Z \oplus W)$.     Consequently, $T/M = T/\mfrak T$ is a free right and left $(B/\mfrak B)$-module with a basis $\{ y^{u_1}+M, \dots, y^{u_s}+M \}$, so $T/M$ is finitely generated as a right $(B/\mfrak B)$-module and all regular elements of $B/\mfrak B$ are also regular in $T/M$. Therefore Lemma \ref{finiteextLd} and Proposition \ref{modfinext}, in combination with \eqref{LdB/mB}, yield
$$
\Ld(T/M) = \gk(T/M) = r.  \qedhere
$$
\end{proof}

In \cite[Theorem 1.6]{bln}, Tdeg-stability is proved for primitive quotients of uniparameter quantum nilpotent algebras. Here, we extend 
the result to all quantum nilpotent algebras.

\begin{theorem} \label{theorem-torsionfree} 
Let $R$ be a quantum nilpotent algebra 
 and let $P$ be a primitive ideal of R. Suppose that $J = J_w$ is the $\ch$-prime ideal 
 of $R$ such that $P\in\spec_J(R)$. Then 
 
{\rm(a)} There is an Ore set $\cf_w \subseteq R/J$ of regular $\ch$-eigenvectors such that $(R/J)\cf_w ^{-1}$ is a quantum torus over $\mk$ and $(P/J)\cf_w^{-1}$ is a maximal ideal of $(R/J)\cf_w^{-1}$.

 {\rm(b)} $R/J$ and $R/P$ are $\Tdeg$-stable.
 
{\rm(c)} $\gk\bigl( (R/J)\cf_w ^{-1} \bigr) = \gk(R/J)$ and $\gk\bigl( (R/P)\cf_w ^{-1} \bigr) = \gk(R/P)$.
\end{theorem}

\begin{proof} (a) By Theorem \ref{matchloc2}, there is an Ore set $\cf_w \subset R/J$ consisting of regular $\ch$-eigenvectors such that
\begin{gather*}
\spec_{J}(R) = \{ P \in \spec(R) \mid P\supseteq J \ \text{and} \ (P/J) \cap \cf_w = \varnothing \}  \\
(R/J)\cf_w^{-1} \cong (\ol{R}/K_w)\ce_w^{-1} \cong \co_{\Lambda_w}((\mk^*)^{N-|w|}),
\end{gather*}
where $\Lambda_w$ is a submatrix of $\Lambda$ (recall \eqref{Rbar.qtorus}). 
In view of \cite[Theorem 4.4]{gl2}, the primitive ideal $P$ is maximal in $\spec_J(R)$, and consequently $(P/J)\cf_w^{-1}$ is a maximal ideal of $(R/J)\cf_w^{-1}$.

(b)(c) We have already shown in Corollary \ref{GKRmodJw} that $R/J$ is $\Tdeg$-stable. Note that the image of $\cf_w$ in $R/P$ consists of regular elements (e.g., \cite[Lemma 10.19]{gw}), so that $(R/P)\cf_w^{-1}$ is naturally isomorphic to a subalgebra of $\Fract R/P$. In view of part (a) and Proposition \ref{qtorusmodmax}, the second part of (b) and part (c) follow from Proposition \ref{transfer.Tdeg}.
\end{proof}




\section{Tauvel's height formula}  

\begin{theorem}
Let $P$ be a prime ideal of the quantum nilpotent algebra $R$. Then 
\[
\gk(R/P) + \height(P)=\gk(R); 
\]
that is, Tauvel's height formula holds for all quantum nilpotent algebras.
\end{theorem} 

\begin{proof} The height formula has been established for the $\ch$-prime ideals of $R$ 
in Theorem~\ref{Tauvel-Hprimes}. 

Next, we deal with the case where $P$ is a primitive ideal. Suppose that $J=J_w$ is the 
$\ch$-prime ideal such that $P\in\spec_J(R)$. Let $\cf_w$ be the Ore set contained 
in $R/J$ that is mentioned in Theorem~\ref{theorem-torsionfree}(a). Since $(R/J)\cf_w^{-1}$ is a quantum torus, its prime spectrum is catenary and Tauvel's height formula holds in this algebra (e.g., \cite[Theorem II.9.14]{bg}).  

Now 
 $\gk\bigl( (R/J)\cf_w^{-1} \bigr) = \gk(R/J)$ and $\gk\bigl((R/P)\cf_w^{-1} \bigr)= \gk(R/P)$ by Theorem~\ref{theorem-torsionfree}(c).
 Hence, 
 \[
\gk(R/P)=\gk\bigl((R/P)\cf_w^{-1} \bigr) =^1 \gk\bigl( (R/J)\cf_w^{-1} \bigr) -\height\bigl( (P/J)\cf_w^{-1} \bigr)
=\gk(R/J) - \height(P/J),
\]
where $(=^1)$ follows from Tauvel's height formula in $(R/J)\cf_w^{-1}$, and
so we obtain $\gk(R/J)=\gk(R/P)+\height(P/J)$.
Consequently, 
\begin{align*}
N &\geq\gk(R/P)+\height(P)\geq \bigl( \gk(R/J)-\height(P/J)\bigl)+
\bigl(\height(P/J)+\height(J)\bigr)  \\
&=\gk(R/J)+\height(J)=N;
\end{align*}
and so $\gk(R/P)+\height(P)=N$, as required.

Finally, let $P$ be an arbitrary prime ideal belonging to the $J$-stratum of $\spec(R)$, and let $Q$ be a 
maximal element of that stratum with $J\subseteq P\subseteq Q$. Then $Q$ is primitive by \cite[Theorem 4.4]{gl2}.

Within the stratum, we have catenarity due to the fact that
\[
\bigl( \spec_J(R), \subseteq \bigr) \cong \bigl( \spec \bigl( (R/J)\cf_w^{-1} \bigr), \subseteq \bigr).
\]
Consequently,
\begin{eqnarray*}
N
&\geq& 
\gk(R/P)+\height(P)\\
&\geq&
\bigl(\gk(R/Q)+\height(Q/P)\bigr)+
\bigl(\height(P/J)+\height(J)\bigr)\\
&=^1&
\gk(R/Q)+\bigl(\height(Q/J)-\height(P/J)\bigr)+
\height(P/J)+\height(J)\\
&=&
\gk(R/Q)+\height(Q/J)+
\height(J)\\
&=^2&
\gk(R/Q)+\bigl(\gk(R/J)-\gk(R/Q)\bigr)+\height(J)\\
&=&
\gk(R/J)+\height(J) =N,
\end{eqnarray*}
where in $(=^1)$ we are using catenarity within the stratum, and $(=^2)$ holds by the 
equality established in the primitive case above. Therefore 
\[
\gk(R/P)+\height(P)=N,
\]
as required.
\end{proof} 


\section{Examples}

We have shown that quantum nilpotent algebras satisfy Tauvel's height formula, but the question as to whether or not they are 
catenary remains open. We have also seen that in the presence of suitable homological conditions, normal separation implies catenarity and Tauvel's height formula. However, for algebras that are not quantum nilpotent algebras, the notion of catenarity and Tauvel's height formula are independent, as we see in the following examples.

\begin{example}{\rm 
Several examples are known of algebras which are catenary but do not satisfy Tauvel's height formula, such as the group algebra of the Heisenberg group (over any field) \cite[Example 3.8]{bell-sig} and the enveloping algebra of $\frak{sl}_2(\mk)$ (for $\mk$ algebraically closed of characteristic zero) \cite[p.411]{bell-sig}. We also point to \cite[Example 2.9]{bell-sig}:
Let $A:=\mk[x,y]$ where $\chr\mk=0$. Let $\delta$ be the $\mk$-linear derivation given by 
$\delta = (2y)\partial/\partial x +(x+y^2)\partial/\partial y$. Set $R:=A[z;\delta]$. 
The ring $R$ has Gelfand-Kirillov dimension three. However, 
the ideal $xR+yR$ is a prime ideal of height one, but $R/(xR+yR)\cong\mk[z]$ has Gelfand-Kirillov dimension one; so 
Tauvel's height formula fails for this ideal. 
This example was originally constructed by Jordan in \cite{jordan}. }
\end{example} 

A modification of the previous example produces an example that is not catenary and does not satisfy Tauvel's height formula.

\begin{example}
{\rm 
Let $A:=\mk[x,y,z]$  where $\chr\mk=0$ and $\delta$ is the $\mk$-linear derivation given by 
$\delta = (2yz)\partial/\partial x +(x+y^2)\partial/\partial y$. Set $R:=A[w;\delta]$. 
See \cite[Example 2.10]{bell-sig} for details. 
}
\end{example} 

It is easily seen that if Tauvel's height formula holds in all prime factors of an algebra $R$, then $R$ is catenary. In fact, it suffices to know that for any prime ideals $P \supsetneq Q$ of $R$ with $\height(P/Q)=1$, the equality $\gk(R/Q) = \gk(R/P)+1$ holds. Namely, this assumption implies that for any prime ideals $P \supsetneq Q$ of $R$, all saturated chains of prime ideals between $P$ and $Q$ have length $\gk(R/Q) - \gk(R/P)$.

It appears unlikely that Tauvel's height formula alone (holding just in an algebra rather than in all prime factors) implies catenarity, but no examples of non-catenary algebras with finite GK-dimension which satisfy Tauvel's height formula are known.

\begin{example}{\rm 
One might imagine that there would be a complementary result to Proposition \ref{largerCauchon} about replacing a white box by a black box. Specifically, if $w$ is a Cauchon diagram for $R$ and $w \ne \gc 1,N \dc$, one could conjecture that there exists a Cauchon diagram $w' \supsetneqq w$ such that $|w'\setminus w| = 1$ and $J_{w'} \supsetneqq J_w$. If such a result did hold then 
applying it iteratively starting
with $w$ empty would yield a chain of $\ch$-primes of length equal to
$\gk(R)$. That fails, e.g.~for the first quantised Weyl algebra, $R = A_1^q(\mk) := \mk \langle x_1,x_2 \mid x_2x_1 - qx_1x_2 = 1 \rangle$, with $q \in \mk^*$ not a root of unity. Here $\gk(R) = 2$ but
there are only two $\ch$-primes altogether, namely $0$ and $J := \langle x_1x_2 - x_2x_1 \rangle$. The Cauchon diagrams of $0$ and $J$ are $\varnothing$ and $\{1\}$, respectively. What goes wrong with the conjectured result is that $\{1,2\}$
is not a Cauchon diagram for $R$.
}
\end{example}


\section*{Acknowledgement}
We thank James Zhang for useful discussions about transcendence degrees.


\medskip


\linespread{1.0}
\selectfont

K R Goodearl:

Department of Mathematics,

University of California,

Santa Barbara, CA 93106, USA

Email: {\tt goodearl@math.ucsb.edu} 

\medskip

S Launois: 

School of Mathematics, Statistics and Actuarial Science,

University of Kent

Canterbury, Kent, CT2 7FS, UK

Email: {\tt S.Launois@kent.ac.uk}

\medskip

T H Lenagan: 

Maxwell Institute for Mathematical Sciences,

School of Mathematics, University of Edinburgh,

James Clerk Maxwell Building, King's Buildings, 

Peter Guthrie Tait Road,

Edinburgh EH9 3FD, Scotland, UK

E-mail: {\tt tom@maths.ed.ac.uk}

\end{document}